\documentclass[british,english]{amsart}
\usepackage[T1]{fontenc}
\usepackage[latin9]{inputenc}
\usepackage{babel}
\usepackage{amsthm}
\usepackage{amsmath}
\usepackage{amssymb}
\usepackage{graphicx}
\usepackage[unicode=true,
 bookmarks=true,bookmarksnumbered=false,bookmarksopen=false,
 breaklinks=false,pdfborder={0 0 1},backref=false,colorlinks=false]
 {hyperref}
\hypersetup{pdftitle={On local non-zero constraints in PDE with analytic coefficients},
 pdfauthor={Giovanni S. Alberti and Yves Capdeboscq}}

\makeatletter
\theoremstyle{plain}
\newtheorem{thm}{\protect\theoremname}
  \theoremstyle{definition}
  \newtheorem{defn}[thm]{\protect\definitionname}
  \theoremstyle{remark}
  \newtheorem{rem}[thm]{\protect\remarkname}
  \theoremstyle{plain}
  \newtheorem{cor}[thm]{\protect\corollaryname}
  \theoremstyle{plain}
  \newtheorem{lem}[thm]{\protect\lemmaname}
  \theoremstyle{plain}
  \newtheorem{prop}[thm]{\protect\propositionname}

\usepackage{tikz}

\usepackage{enumitem}
\setlist{leftmargin=*}

\usepackage{color}

\makeatother

  \addto\captionsbritish{\renewcommand{\corollaryname}{Corollary}}
  \addto\captionsbritish{\renewcommand{\definitionname}{Definition}}
  \addto\captionsbritish{\renewcommand{\lemmaname}{Lemma}}
  \addto\captionsbritish{\renewcommand{\propositionname}{Proposition}}
  \addto\captionsbritish{\renewcommand{\remarkname}{Remark}}
  \addto\captionsbritish{\renewcommand{\theoremname}{Theorem}}
  \addto\captionsenglish{\renewcommand{\corollaryname}{Corollary}}
  \addto\captionsenglish{\renewcommand{\definitionname}{Definition}}
  \addto\captionsenglish{\renewcommand{\lemmaname}{Lemma}}
  \addto\captionsenglish{\renewcommand{\propositionname}{Proposition}}
  \addto\captionsenglish{\renewcommand{\remarkname}{Remark}}
  \addto\captionsenglish{\renewcommand{\theoremname}{Theorem}}
  \providecommand{\corollaryname}{Corollary}
  \providecommand{\definitionname}{Definition}
  \providecommand{\lemmaname}{Lemma}
  \providecommand{\propositionname}{Proposition}
  \providecommand{\remarkname}{Remark}
\providecommand{\theoremname}{Theorem}

\begin{document}
\global\long\def\R{\mathbb{R}}

\global\long\def\N{\mathbb{N}}

\global\long\def\C{\mathbb{C}}

\global\long\def\Cl{C}

\global\long\def\tr{\mathrm{tr}}

\global\long\def\z{\theta}

\global\long\def\Se{\mathcal{S}}

\global\long\def\P{L_{\beta_{1},\beta_{2}}^{\infty}(\Omega)}

\global\long\def\V{H_{0}^{1}(\Omega;\C)}

\global\long\def\Vp{H^{-1}(\Omega;\C)}

\global\long\def\Vr{H_{0}^{1}(\Omega)}

\global\long\def\Honer{H^{1}(\Omega)}

\global\long\def\Hone{H^{1}(\Omega;\C)}

\global\long\def\Hhalf{H^{1/2}(\Omega;\C)}

\global\long\def\Hhalfr{H^{1/2}(\Omega)}

\global\long\def\phi{\varphi}

\global\long\def\epsilon{\varepsilon}

\global\long\def\div{{\rm div}}

\global\long\def\ld{L^{2}(\Omega;\C^{3})}

\global\long\def\ldr{L^{2}(\Omega;\R^{3})}

\global\long\def\linf{L^{\infty}(\Omega;\C)}

\global\long\def\linfr{L^{\infty}(\Omega;\R)}

\global\long\def\Conealp{\mathcal{C}^{1,\alpha}(\overline{\Omega};\C)}

\global\long\def\Conealpr{\mathcal{C}^{1,\alpha}(\overline{\Omega})}

\global\long\def\Calpr{\mathcal{C}^{0,\alpha}(\overline{\Omega};\R^{d\times d})}

\global\long\def\Calprsca{\mathcal{C}^{0,\alpha}(\overline{\Omega})}

\global\long\def\Czeroonescal{\mathcal{C}^{0,1}(\overline{\Omega})}

\global\long\def\Calpvect{\mathcal{C}^{0,\alpha}(\overline{\Omega};\R^{d})}

\global\long\def\Cone{\Cl^{1}(\overline{\Omega};\C)}

\global\long\def\Coner{\Cl^{1}(\overline{\Omega})}

\global\long\def\Kad{K_{ad}}

\global\long\def\mina{\lambda}

\global\long\def\maxa{\Lambda}

\global\long\def\Hcurl{H(\curl,\Omega)}

\global\long\def\Hmu{H^{\mu}(\curl,\Omega)}

\global\long\def\Hocurl{H_{0}(\curl,\Omega)}

\global\long\def\Hdiv{H(\div,\Omega)}

\global\long\def\k{\omega}

\global\long\def\div{{\rm div}}

\global\long\def\curl{{\rm curl}}

\global\long\def\supp{{\rm supp}}

\global\long\def\sp{{\rm span}}

\global\long\def\ii{\mathbf{i}}

\global\long\def\E{E_{\k}}

\global\long\def\Ei{E_{\k}^{i}}

\global\long\def\Hi{H_{\k}^{i}}

\global\long\def\H{H_{\k}}

\global\long\def\EE{\tilde{E}_{\k}}

\global\long\def\bo{\partial\Omega}

\global\long\def\Co{\Cl(\overline{\Omega};\C^{3})}

\global\long\def\ve{\theta}

\global\long\def\so{\hat{\sigma}}

\global\long\def\order{\kappa}

\global\long\def\p{C}

\global\long\def\pert{t}

\global\long\def\e{\mathbf{e}}

\global\long\def\epsi{\varepsilon}

\global\long\def\A{\mathcal{A}}

\title{On local non-zero constraints in PDE with analytic coefficients}
\author{Giovanni S. Alberti}
\address{D\'epartement  Math\'ematiques et Applications,  \'Ecole Normale Sup\'erieure, 45 rue d'Ulm, 75005 Paris, France.} 
\email{giovanni.alberti@ens.fr}
\author{Yves Capdeboscq}
\address{Mathematical Institute, University of Oxford, Oxford OX2 6GG, UK.} 
\email{yves.capdeboscq@maths.ox.ac.uk}
\date{7\textsuperscript{th} January 2015}

\thanks{The authors have benefited from the support of the EPSRC Science \& Innovation Award to the Oxford Centre for Nonlinear PDE (EP/EO35027/1).} 


\begin{abstract}
We consider the Helmholtz equation with real analytic coefficients
on a bounded domain $\Omega\subset\R^{d}$. We take $d+1$ prescribed
boundary conditions $f^{i}$ and frequencies $\omega$ in a fixed
interval $[a,b]$. We consider a constraint on the solutions $u_{\omega}^{i}$
of the form $\zeta(u_{\omega}^{1},\ldots,u_{\omega}^{d+1},\nabla u_{\omega}^{1},\ldots,\nabla u_{\omega}^{d+1})\neq0$,
where $\zeta$ is analytic, which is satisfied in $\Omega$ when $\omega=0$.
We show that for any $\Omega^{\prime}\Subset\Omega$ and almost any
$d+1$ frequencies $\omega_{k}$ in $[a,b]$, there exist $d+1$ subdomains
$\Omega_{k}$ such that $\Omega^{\prime}\subset\cup_{k}\Omega_{k}$
and $\zeta(u_{\omega_{k}}^{1},\ldots,u_{\omega_{k}}^{d+1},\nabla u_{\omega_{k}}^{1},\ldots,\nabla u_{\omega_{k}}^{d+1})\neq0$
in $\Omega_{k}$. This question comes from hybrid imaging inverse
problems. The method used is not specific to the Helmholtz model and
can be applied to other frequency dependent problems.
\end{abstract}
\maketitle
\section{Introduction}

The motivation for this work comes from hybrid, or multi-physics,
parameter identification problem in boundary value problems for partial
differential equations \cite{bal2012_review,alberti-capdeboscq-2014}.
In such imaging modalities, one part of the inverse problem can be
described in general terms as follows. Suppose that $u_{\omega}^{1},u_{\omega}^{2},\ldots,u_{\omega}^{N}$
are the solutions of a partial differential equation of the form
\[
\left\{ \begin{array}{l}
P(x,\omega,u_{\omega}^{i})=0\quad\text{in \ensuremath{\Omega},}\\
u_{\omega}^{i}=f^{i}\quad\text{on \ensuremath{\partial\Omega},}
\end{array}\right.
\]
for $i=1,\ldots,N.$ Suppose further that the parameter $\omega$
is known, $f^{i}$ is known, and some pointwise information is known
about a functional of the solutions, e.g. $H(u_{\omega}^{1},\ldots,u_{\omega}^{N},\nabla u_{\omega}^{1},\ldots,\nabla u_{\omega}^{N})$,
in $\Omega^{\prime}$, a subdomain of $\Omega$, or possibly in all
of $\Omega$. The problem is to reconstruct spatial dependence of
the operator $P$ from this information. In this article we focus
on a particular model, a problem of Helmholtz type, in a smooth bounded
domain $\Omega\subseteq\R^{d}$, $d\ge2$, given by
\begin{equation}
\left\{ \begin{array}{l}
-\div(a\,\nabla u)-(\omega^{2}\epsi+\ii\omega\sigma)\, u=0\qquad\text{in \ensuremath{\Omega},}\\
u=f\qquad\text{on \ensuremath{\partial\Omega}.}
\end{array}\right.\label{eq:combined-i-helmholtz-1}
\end{equation}
We assume that $a\in L^{\infty}\left(\Omega;\R^{d\times d}\right)$
and that $a$ is symmetric and uniformly positive definite and bounded,
that is, for all $\xi\in\mathbb{R}^{d}$ there holds 
\begin{equation}
\lambda^{-1}\left|\xi\right|^{2}\le\xi\cdot a\xi\le\lambda\left|\xi\right|^{2}\quad\text{ a.e. in }\Omega\label{eq:ellipticity_a-multi}
\end{equation}
for some positive constant $\lambda$, whereas $\epsi,\sigma\in\linfr$
satisfy
\begin{equation}
\lambda^{-1}\le\epsi\le\lambda,\quad0\le\sigma\le\lambda\quad\mbox{\text{ a.e. in }\ensuremath{\Omega}}.\label{eq:bounds_epsilon_multi}
\end{equation}

Assumptions \eqref{eq:ellipticity_a-multi} and \eqref{eq:bounds_epsilon_multi}
guarantee that problem \eqref{eq:combined-i-helmholtz-1} has a unique
solution in $H^{1}(\Omega;\C)$ for every $f\in H^{\frac{1}{2}}(\partial\Omega;\R)$
and $\k\in D=\C\setminus\Sigma$, where $\Sigma$ denotes the set
of the discrete Dirichlet eigenvalues of the problem. Let $\A=[A_{\min},A_{\max}]$
represent the frequencies we have access to, for some $0<A_{\min}<A_{\max}$.
For simplicity, we suppose that $\A\subseteq D$. 

Given several boundary conditions $f^{i}\in H^{\frac{1}{2}}(\partial\Omega;\R)$
and frequencies $\omega_{k}\in\A$, the pointwise information available
(or observable) in this case could be
\[
u_{\omega_{k}}^{i},\, a\nabla u_{\omega_{k}}^{i},\mbox{ or }qu_{\omega_{k}}^{i}u_{\omega_{k}}^{j}\mbox{ or }a\nabla u_{\omega_{k}}^{i}\cdot\nabla u_{\omega_{k}}^{j},
\]
for some $i,j$ and $k$. As all these data come from measurements,
it is important to know a priori that the numbers obtained are not
mostly measurement error or background noise: we want to ensure that
the modulus of these data is non-zero. For instance, given $\omega$
and a boundary condition $f^{1}$, we want to ensure that 
\begin{equation}
u_{\omega}^{1}\neq0.\label{eq:const-F1}
\end{equation}
Alternatively, combining multiple data and constraints into one functional,
given $\omega$ and $d+1$ boundary conditions $f^{1},\ldots,f^{d+1}$
we write 
\begin{equation}
F(\omega,f^{1},\ldots,f^{d+1})=u_{\omega}^{1}\det\left[\begin{array}{ccc}
u_{\omega}^{1} & \ldots & u_{\omega}^{d+1}\\
\\
\nabla u_{\omega}^{1} & \ldots & \nabla u_{\omega}^{d+1}
\end{array}\right],\label{eq:def-F}
\end{equation}
and we want to ensure that 
\begin{equation}
F(\omega,f^{1},\ldots,f^{d+1})\neq0\quad\text{in }\Omega'.\label{eq:const-F}
\end{equation}
The constraints \eqref{eq:const-F1} and \eqref{eq:const-F}, or related
quantities, appear in \cite{cap2011,ammari2012quantitative,bal2011reconstruction,triki2010,alberti-capdeboscq-2014}.
If $\omega$ is large, that is, greater than a constant depending
on $\lambda$ and $\Omega$ only, \eqref{eq:const-F1} (and a fortiori
\eqref{eq:const-F}) cannot be satisfied in the whole domain, as for
any boundary condition in $H^{\frac{1}{2}}(\partial\Omega;\R)$, the
field $u_{\omega}^{1}$ must cancel (at least when $\sigma=0$). Thus
multiple boundary conditions or frequencies must be considered. 
\begin{defn}
Given a finite set of frequencies $\left\{ \omega_{1},\ldots,\omega_{K}\right\} \in\A^{K}$
and a finite set of boundary conditions $f^{1},\dots,f^{N}\in H^{\frac{1}{2}}(\partial\Omega;\mathbb{R})$,
we say that $\left\{ \omega_{1},\ldots,\omega_{K}\right\} \times\{f^{1},\dots,f^{N}\}$
is a \emph{set of measurements}.
\end{defn}
The first concern is whether there exists a set of measurements such
that \eqref{eq:const-F} is satisfied everywhere by a subset of this
set. The precise meaning of this statement is given by the following
definition. 
\begin{defn}
\label{def:zeta-complete-1} Take $\Omega'\subseteq\Omega$. Given
$K,N\in\mathbb{N}^{*}$, a set of measurements $\left\{ \omega_{1},\ldots,\omega_{K}\right\} \times\{f^{1},\dots,f^{N}\}$
is \emph{$F$-complete in $\Omega'$} if there exists an open cover
of $\Omega'$
\[
\Omega'=\bigcup_{p=1}^{P}\Omega'_{p},
\]
such that for each $p$ there exist $k\in\{1,\ldots,K\}$ and $i_{1},\ldots,i_{d+1}\in\{1,\ldots,N\}$
such that 
\begin{equation}
\bigl|F\bigl(\omega_{k},f^{i_{1}}\dots,f^{i_{d+1}}\bigr)(x)\bigr|>0,\qquad x\in\Omega'_{p}.\label{eq:zeta-complete-1}
\end{equation}

\end{defn}
In other words, a $F$-complete set of measurements gives a cover
of the domain $\Omega'$ into a finite collection of subdomains, such
that the constraints \eqref{eq:zeta-complete-1} are satisfied in
each subdomain for different frequencies and boundary conditions.

Several results show the existence of such $F$-complete sets. In
the single-frequency case, namely with $K=1$, the existence of such
$F$-complete sets can be proved by using Complex Geometric Optics
(CGO) solutions \cite{CALDERON-1980,SYLVESTER-UHLMANN-87,bal2010inverse}
or the Runge approximation property \cite{bal2011reconstruction,bal2011quantitative,bal2013reconstruction}
under appropriate regularity hypotheses on $a,\epsilon$ and $\sigma$.
When using CGO, only $P=2$ subdomains are needed, while with the
Runge approximation approach $P$ is larger than 2. These approaches
do not indicate how suitable boundary conditions should be chosen
in practice, as the proof of their existence relies on the unknown
coefficients: without additional a priori information, the search
for these boundary conditions requires many trials. Moreover, since
CGO are very oscillatory, they may require a very sophisticated practical
apparatus to be implemented.

An alternative method consists in fixing the boundary conditions and
varying the frequency.
\begin{thm}
\label{fac:two} Take $d=2$ and suppose that $\Omega$ is convex
and $a\in C^{0,\alpha}$ for some $\alpha>0$. There exist $K\in\mathbb{N}^{*}$
and $C>0$ depending only on $\Omega$, $\Omega^{\prime}$, $\lambda$,
$\alpha$, $A_{\min}$ and $A_{\max}$ such that
\[
\{\omega_{k}=A_{\min}+\left(A_{\max}-A_{\min}\right)\frac{k-1}{K-1}:k=1,\ldots,K\}\times\{1,x_{1},x_{2}\}
\]
is $F$-complete in $\Omega^{\prime}$. More precisely, there exists
an open cover $\Omega'=\bigcup_{k=1}^{K}\Omega'_{k}$ such that 
\[
\left|F(\omega_{k},1,x_{1},x_{2})\right|\ge C\mbox{ in }\Omega'_{k}.
\]
This result also holds when $d=3$ with the boundary conditions $1,x_{1},x_{2},x_{3}$
provided that $\left\Vert a-I_{d}\right\Vert _{C^{0,\alpha}}\leq\delta$
where $\delta$ depends only on $\Omega$, $\lambda$ and $\alpha$.
\end{thm}
This result is proved in \cite{2014-albertigs} (see also \cite{alberti2013multiple,albertigsII}).
In this result, $K$ is bounded a priori and possibly large, but the
boundary conditions are fixed a priori and non-oscillatory. The proof
of Theorem~\ref{fac:two} relies on an analytic continuation argument
with respect to the frequency, immersed in the complex plane. When
$\omega=0$, problem \eqref{eq:combined-i-helmholtz-1} becomes 
\[
\left\{ \begin{array}{l}
-\div(a\,\nabla u_{0})=0\qquad\text{in \ensuremath{\Omega},}\\
u_{0}=f\qquad\text{on \ensuremath{\partial\Omega}.}
\end{array}\right.
\]
For the boundary condition $f^{1}=1$, the solution is simply $u_{0}^{1}\equiv1$.
In two dimensions, for the boundary conditions $f^{2}=x_{1}$ and
$f^{3}=x_{2}$, since $\Omega$ is convex, it is known that $\mbox{det}\left(\nabla u_{0}^{2},\nabla u_{0}^{3}\right)>0$
in $\Omega$ \cite{alessandrinimagnanini1994,bauman2001univalent,ALESSANDRINI-NESI-01}.
Therefore 
\begin{equation}
F(0,1,x_{1},x_{2})=1\times\det\left[\begin{array}{ccc}
1 & u_{0}^{2} & u_{0}^{3}\\
\\
0 & \nabla u_{0}^{2} & \nabla u_{0}^{3}
\end{array}\right]>0\mbox{ in }\Omega.\label{eq:2D}
\end{equation}
In higher dimension such a result is not available, however Schauder
elliptic regularity theory shows that if 
\begin{equation}
\|a-I_{d}\|_{C^{0,\alpha}}\leq\delta\label{eq:small-condition}
\end{equation}
for some $\delta>0$ small enough, then 
\begin{equation}
F(0,1,x_{1},\ldots,x_{d})=1\times\det\left[\begin{array}{cccc}
1 & u_{0}^{2} & \ldots & u_{0}^{d+1}\\
\\
0 & \nabla u_{0}^{2} &  & \nabla u_{0}^{d+1}
\end{array}\right]>0\mbox{ in }\Omega.\label{eq:3D}
\end{equation}
The principle of the proof of Theorem~\ref{fac:two} is then to show
that this positivity property can be transported to any interval $\mathcal{A}$
in a predictable manner \cite{2014-albertigs}. 

The goal of this article is to investigate what is the minimal number
of required frequencies $K$ (or, equivalently, the number $P$ of
subdomains in Definition~\ref{def:zeta-complete-1}) in dimension
$d\geq2$, for the fixed boundary conditions $1,x_{1},\ldots,x_{d}$.
We consider a technically very convenient particular case, namely
we assume that
\begin{equation}
a,\epsi\mbox{, and }\sigma\mbox{ are real analytic in }\Omega.\label{eq:analytic}
\end{equation}
The main result of the paper reads as follows.
\begin{thm}
\label{thm:d+1 frequencies-helmholtz}Assume that \eqref{eq:ellipticity_a-multi},
\eqref{eq:bounds_epsilon_multi} and \eqref{eq:analytic} hold true.
Suppose that $\A\subseteq D$ and take $F$ as in \eqref{eq:def-F},
$\Omega'\Subset\Omega$ and $f^{1},\dots,f^{d+1}\in H^{\frac{1}{2}}(\partial\Omega;\mathbb{R})$.
If
\begin{equation}
F\left(0,f^{1},\dots,f^{d+1}\right)(x)\neq0,\qquad x\in\Omega,\label{eq:zeta-helm-assumtion 0}
\end{equation}
then
\[
\bigl\{(\k_{1},\dots,\k_{d+1})\in\A^{d+1}:\{\k_{k}\}_{k}\times\{f^{1},\ldots,f^{d+1}\}\text{ is }F\text{-complete in \ensuremath{\Omega'}}\bigr\}
\]
is open and dense in $\A^{d+1}$.\end{thm}
\begin{rem}
This argument can be used for many other partial differential equations,
see \cite{2014-albertigs} and \cite{alberti2013multiple,albertigsII,albertidphil,alberti-ammari-ruan-2014}
for variants of this argument. It applies in particular to the anisotropic
Maxwell system of equations.
\end{rem}

\begin{rem}
There is nothing special about the function $F$, which was used as
an illustration in this paper. We only use that $F(\omega,1,x_{1},\ldots,x_{d})$
expressed in terms of $u_{\omega}^{1},\ldots,u_{\omega}^{d+1}$, that
is, 
\[
F(\omega,1,x_{1},\ldots,x_{d})=\zeta\left(u_{\omega}^{1},\ldots,u_{\omega}^{d+1},\nabla u_{\omega}^{1},\ldots,\nabla u_{\omega}^{d+1}\right)
\]
where 
\[
\zeta\left(z_{1},\ldots,z_{d+1},\xi_{1},\ldots,\xi_{d+1}\right)=z_{1}\det\left(\left[\begin{array}{ccc}
z_{1} & \ldots & z_{d+1}\\
\xi_{1} & \ldots & \xi_{d+1}
\end{array}\right]\right),
\]
is holomorphic (complex analytic). Any other constraint that can be
written in terms of a holomorphic function $\zeta$ can be substituted
to $F$. 
\end{rem}
In view of \eqref{eq:2D} and \eqref{eq:3D}, the following conclusion
then naturally follows.
\begin{cor}
\label{cor:example} Assume that \eqref{eq:ellipticity_a-multi},
\eqref{eq:bounds_epsilon_multi} and \eqref{eq:analytic} hold true.
When $d=2$ assume that $\Omega$ is convex and when $d\geq3$ assume
that \eqref{eq:small-condition} holds. Suppose that $\A\subseteq D$
and take $\Omega'\Subset\Omega$. Then 
\[
\bigl\{(\k_{1},\dots,\k_{d+1})\in\A^{d+1}:\{\k_{k}\}_{k}\times\{1,x_{1}\ldots,x_{d}\}\text{ is }F\text{-complete in \ensuremath{\Omega'}}\bigr\}
\]
is open and dense in $\A^{d+1}$.\end{cor}
\begin{rem}
\label{rem:d+1-frequencies-helmholtz-C}This result shows that almost
any $d+1$ frequencies in $\A$ give a F-complete set. An a priori
estimate on the lower bound $C$ for $F$ (as in Theorem~\ref{fac:two})
cannot be obtained for arbitrary frequencies in an open and dense
set in $\A^{d+1}$, as this bound tends to zero when the frequencies
are chosen near the residual set.
\end{rem}
It is easy to verify that this result is optimal in dimension $1$
and $2$ numerically. To intuitively see that $K=d+1$ is natural
in any dimension, consider the level set $u_{\omega}^{1}=0$, for
a given $\omega$. It is a priori a $d-1$ dimensional object. If
we consider the intersection of two such level sets for $\omega_{1}$
and $\omega_{2}$, we expect the resulting object to be $d-2$ dimensional,
the intersection of $d$ such level sets to be zero dimensional, i.e.
discrete, and the intersection of $d+1$ level sets to be empty. Figure~\ref{fig:sketch}
is a graphical illustration of this idea.

\begin{figure}
\centering{}\includegraphics[width=0.8\columnwidth]{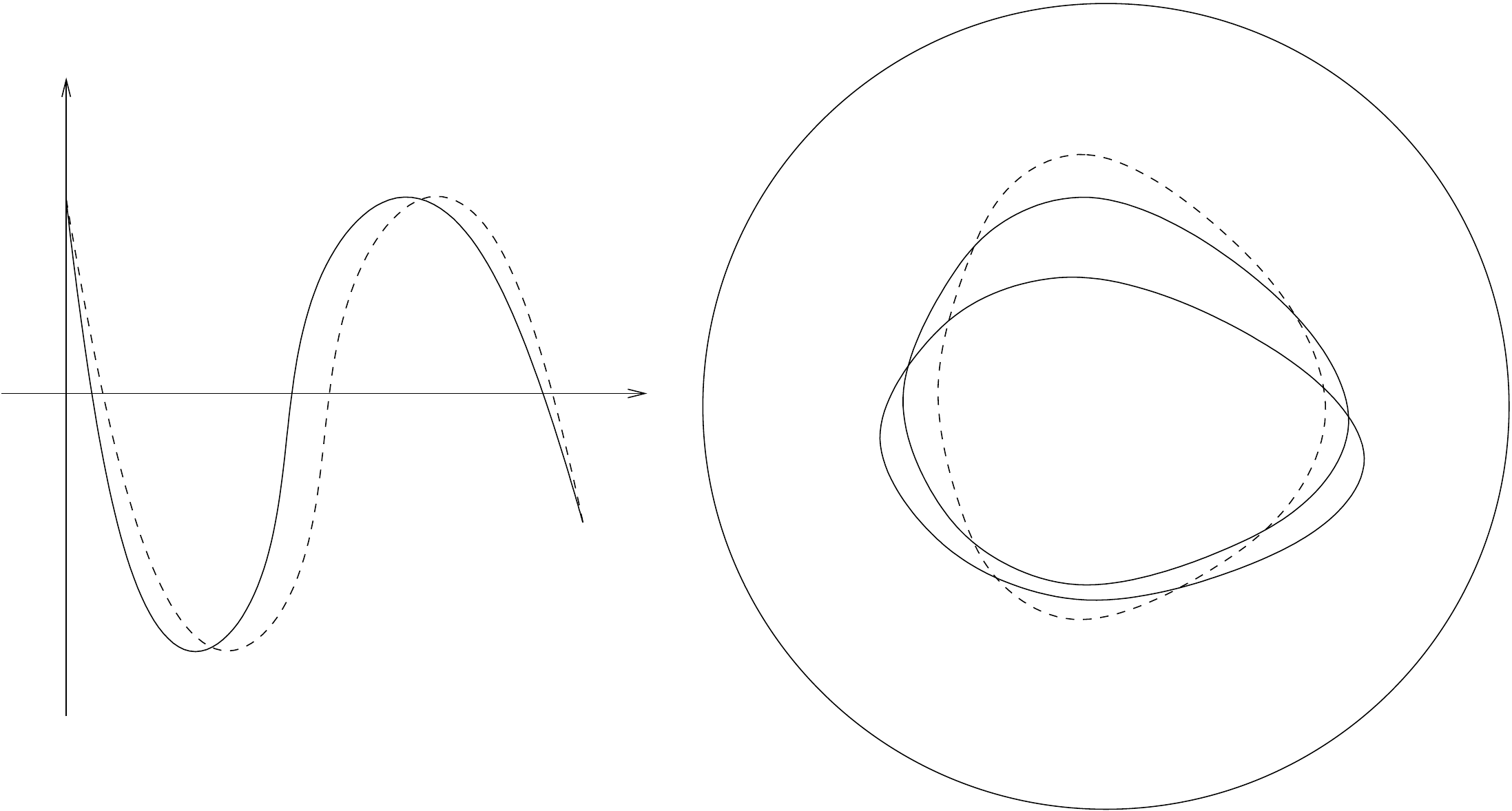}\caption{\label{fig:sketch}A sketch of the $d+1$ rationale in dimension $1$
and $2$. On the left, we represent two solutions with the same boundary
data when $d=1$: the locus $\{u_{\omega}=0\}$ moves with the parameter
$\omega$, and the intersection of two zero sets is empty. On the
right, we represent three zero level sets of the function $F$ in
two dimensions. The outer circle represents the boundary of $\Omega$.
When only two values of $\omega$ are used, the zero level sets still
contain common points, but the intersection of the three level sets
is empty.}
\end{figure}

The rest of the paper is devoted to the proof of Theorem~\ref{thm:d+1 frequencies-helmholtz}.

\section{Proof of Theorem~\ref{thm:d+1 frequencies-helmholtz}}

First recall that the analyticity of the coefficients implies the
analyticity of the solutions. Let $C^{A}(\Omega;\C)$ denote the space
of complex-valued real analytic maps over $\Omega$.
\begin{lem}[\selectlanguage{british}%
\cite{morrey-nirenberg-1957}\selectlanguage{english}%
]
\label{lem:helmholtz-analytic-regularity}Assume that \eqref{eq:ellipticity_a-multi},
\eqref{eq:bounds_epsilon_multi} and \eqref{eq:analytic} hold true.
If $\k\in D$ and $f^{i}\in H^{\frac{1}{2}}(\partial\Omega;\mathbb{R})$
then $u_{\omega}^{i}\in C^{A}(\Omega;\C)$.
\end{lem}
Theorem~\ref{thm:d+1 frequencies-helmholtz} will be a consequence
of the following result. 
\begin{prop}
\label{prop:d+1 frequencies-general}Let $\Omega,\Omega'\subseteq\R^{d}$
be smooth domains such that $\Omega'\Subset\Omega$. Let $D\subseteq\C$
be an open set such that $0\in D$ and $\A\subseteq D$. Consider
a map
\[
\z\colon D\to C^{A}(\Omega;\C),\qquad\k\mapsto\z_{\k}
\]
such that for all $x\in\Omega$, $\k\in D\mapsto\z_{\k}(x)\in\C$
is holomorphic and $\z_{0}(x)\neq0$. The set
\[
\Bigl\{(\k_{1},\dots,\k_{d+1})\in\A^{d+1}:\min_{\overline{\Omega'}}(\left|\z_{\k_{1}}\right|+\cdots+\left|\z_{\k_{d+1}}\right|)>0\Bigr\}
\]
is open and dense in $\A^{d+1}$.
\end{prop}
First, let us see why Theorem~\ref{thm:d+1 frequencies-helmholtz}
follows from this result.
\begin{proof}[\selectlanguage{british}%
Proof of Theorem~\ref{thm:d+1 frequencies-helmholtz}\selectlanguage{english}%
]
 Consider the map 
\[
\theta\colon\omega\mapsto F\left(\omega,f^{1},\dots,f^{d+1}\right)=u_{\omega}^{1}\det\left(\left[\begin{array}{ccc}
u_{\omega}^{1} & \ldots & u_{\omega}^{d+1}\\
\\
\nabla u_{\omega}^{1} & \ldots & \nabla u_{\omega}^{d+1}
\end{array}\right]\right).
\]
In view of Lemma~\ref{lem:helmholtz-analytic-regularity}, $\theta_{\omega}\in\Cl^{A}(\Omega;\C)$.
By the general fact that for any $i$ and $x$ the map $\ensuremath{\omega\in D\mapsto(u_{\omega}^{i}(x),\nabla u_{\omega}^{i}(x))}\in\C^{d+1}$
is holomorphic \cite{2014-albertigs}, the map $\k\in D\mapsto\theta_{\k}(x)\in\C$
is holomorphic for all $x\in\Omega$. Moreover, $\theta_{0}(x)\neq0$
for all $x\in\Omega$ by \eqref{eq:zeta-helm-assumtion 0}.

We can apply Proposition~\ref{prop:d+1 frequencies-general} and
obtain that
\[
\Bigl\{(\k_{1},\dots,\k_{d+1})\in\A^{d+1}:\min_{\overline{\Omega'}}(\left|\z_{\k_{1}}\right|+\cdots+\left|\z_{\k_{d+1}}\right|)>0\Bigr\}
\]
is open and dense in $\A^{d+1}$. Note that the condition $\min_{\overline{\Omega'}}\sum_{k=1}^{d+1}\bigl|\theta_{\omega_{k}}\bigr|>0$
is equivalent to
\[
\text{for all }x\in\overline{\Omega'}\text{ there exists \ensuremath{k\;}such that \ensuremath{\bigl|F\left(\omega_{k},f^{1},\ldots f^{d+1}\right)\bigr|>0,}}
\]
which means that $\{\k_{k}\}_{k}\times\{f^{i}\}_{i}$ is a $F$-complete
set of measurements in $\Omega'$. Indeed, defining 
\[
\Omega'_{k}=\{x\in\Omega'\,:\,\bigl|F\left(\omega_{k},f^{1},\ldots f^{d+1}\right)\bigr|>0\},
\]
we have $\Omega'=\cup_{k}\Omega'_{k}$. This concludes the proof.
\end{proof}
The rest of this section is devoted to the proof of Proposition~\ref{prop:d+1 frequencies-general},
which is based on the structure of analytic varieties.

An \emph{analytic variety in $\Omega$} is the set of common zeros
of a finite collection of real analytic functions in $\Omega$, namely
$\{x\in\Omega:g_{1}(x)=\cdots=g_{N}(x)=0\}$, for some $g_{1},\dots,g_{N}\in C^{A}(\Omega;\C)$.
For $\k_{1},\dots,\k_{N}\in\A$ we shall consider the analytic variety
\[
Z(\k_{1},\dots,\k_{N})=\{x\in\Omega:\z_{\k_{1}}(x)=\cdots=\z_{\k_{N}}(x)=0\}=\bigcap_{i=1}^{N}Z(\k_{i}).
\]
Analytic varieties can be stratified into submanifolds of different
dimensions.
\begin{lem}[\cite{Whitney-1965}]
\label{lem:stratification}Let $X$ be an analytic variety in $\Omega$.
There exists a locally finite collection $\{A_{l}\}_{l}$ of pairwise
disjoint connected analytic submanifolds of $\Omega$ (satisfying
Whitney's conditions) such that
\[
X=\bigcup_{l}A_{l}.
\]

\end{lem}
The decomposition $X=\cup_{l}A_{l}$ is called \emph{a Whitney stratification}
of $X$. With this in mind, we can define the dimension of an analytic
variety $X=\cup_{l}A_{l}$ by
\begin{equation}
\dim X:=\max_{l}\dim A_{l}.\label{eq:dimension}
\end{equation}

The main result leading to the proof of Proposition~\ref{prop:d+1 frequencies-general}
is the following
\begin{lem}
\label{lem:finite set}Under the hypotheses of Proposition~\ref{prop:d+1 frequencies-general},
let $\Omega''$ be a smooth domain such that $\Omega''\Subset\Omega$
and $X$ be an analytic variety in $\Omega$ such that $X\cap\Omega''\neq\emptyset$.
Then the set 
\[
\{\k\in\A:\dim(Z(\k)\cap X\cap\Omega'')=\dim(X\cap\Omega'')\}
\]
is finite.\end{lem}
\begin{proof}
By contradiction, suppose that the set is infinite. Since $\A$ is
compact, there exist $\k_{n},\k\in\A$, $\k_{n}\to\k$ such that $\dim(Z(\k_{n})\cap X\cap\Omega'')=\dim(X\cap\Omega'')$
and $\k_{n}\neq\k$ for all $n\in\N$. 

Therefore, in view of \eqref{eq:dimension}, for each $n$ there exists
a non-empty connected analytic submanifolds $S_{n}$ such that
\begin{equation}
S_{n}\subseteq Z(\k_{n})\cap X\cap\Omega''\label{eq:first}
\end{equation}
and
\begin{equation}
\dim S_{n}=\dim(X\cap\Omega'').\label{eq:second}
\end{equation}
Choose an arbitrary $x_{n}\in S_{n}$ for all $n\in\N$. Up to a subsequence,
we have $x_{n}\to x$, for some $x\in\overline{\Omega''}$. By Lemma~\ref{lem:stratification}
applied to $X$, there exists an open neighborhood $U$ of $x$ in
$\Omega$ and a finite collection $\{A_{l}\}_{l}$ of analytic submanifolds
of $\Omega$ such that
\[
X\cap U=\cup_{l}A_{l}.
\]
Moreover, since $x_{n}\in S_{n}$ and $x_{n}\to x$, up to a subsequence
we have $S_{n}\cap U\neq\emptyset$ for all $n\in\N$. As $S_{n}\subseteq X$,
up to a subsequence (and relabeling the collection $A_{l}$) we have
for all $n$ 
\begin{equation}
S_{n}\cap U\subset A_{1}.\label{eq:third}
\end{equation}
Since by \eqref{eq:second}, $\dim(S_{n}\cap U)=\dim(X\cap U)$, and
$A_{1}\subset X\cap U$, this implies 
\begin{equation}
\dim(S_{n}\cap U)=\dim A_{1}.\label{eq:forth}
\end{equation}
In view of \eqref{eq:first} we have $\z_{\k_{n}}(y)=0$ for all $y\in S_{n}$.
Therefore, by \eqref{eq:third}, \eqref{eq:forth}, \cite[Theorem 1.2]{ENCISO-PERALTA-2013}
and $\z_{\k}\in C^{A}(\Omega;\C)$ we obtain $S_{n}\cap U=A_{1}$,
whence 
\[
A_{1}\subseteq Z(\k_{n}),\qquad n\in\N.
\]
As $x_{n}\in A_{1}$ for all $n$, we have $x\in\overline{A_{1}}$.
Thus, since $Z(\k_{n})$ is closed, we infer that $x\in Z(\k_{n})$
for all $n\in\N$, namely $\z_{\k_{n}}(x)=0$ for all $n\in\N$. Since
$\k\mapsto\z_{\k}(x)$ is holomorphic, this implies $\z_{0}(x)=0$,
which contradicts the assumptions.
\end{proof}
We are now in a position to prove Proposition~\ref{prop:d+1 frequencies-general}.
\begin{proof}[Proof of Proposition~\foreignlanguage{british}{\ref{prop:d+1 frequencies-general}}]
 Since the map $(\k_{1},\dots,\k_{d+1})\mapsto\min_{\overline{\Omega'}}(\left|\z_{\k_{1}}\right|+\cdots+|\z_{\k_{d+1}}|)$
is continuous, the set $G=\{(\k_{1},\dots,\k_{d+1})\in\A^{d+1}:\min_{\overline{\Omega'}}(\left|\z_{\k_{1}}\right|+\cdots+|\z_{\k_{d+1}}|)>0\}$
is open. It remains to show that $G$ is dense in $\A^{d+1}$.

Take $(\tilde{\k}_{1},\dots,\tilde{\k}_{d+1})\in\A^{d+1}$ and $\epsilon>0$.
Let $\Omega''$ be such that $\Omega'\Subset\Omega''\Subset\Omega$.
We equip $\A^{d+1}$ with the norm
\[
\left\Vert (\k_{1},\dots,\k_{d+1})\right\Vert =\max_{k}\left|\k_{k}\right|.
\]
We now want to construct an element $(\k_{1},\dots,\k_{d+1})\in G$
such that
\begin{equation}
\bigl\Vert(\k_{1},\dots,\k_{d+1})-(\tilde{\k}_{1},\dots,\tilde{\k}_{d+1})\bigr\Vert<\epsilon.\label{eq:close}
\end{equation}
If $\dim Z(\k_{1})\le d-1$, set $\k_{1}=\tilde{\k}_{1}$; obviously
we have $|\k_{1}-\tilde{\k}_{1}|<\epsilon$. If $\dim Z(\k_{1})=d$,
by Lemma~\ref{lem:finite set} we obtain that the set
\[
\{\k\in\A:\dim(Z(\k)\cap\Omega'')=d\}
\]
is finite. Therefore, we can choose $\k_{1}\in\A$ such that
\[
\dim(Z(\k_{1})\cap\Omega'')\le d-1
\]
and $|\k_{1}-\tilde{\k}_{1}|<\epsilon$. Suppose now that we have
constructed $\k_{1},\dots,\k_{k}$ such that $|\k_{j}-\tilde{\k}_{j}|<\epsilon$
for all $j=1,\dots,k$. Let us describe how to construct $\k_{k+1}$.
If $Z(\k_{1},\dots,\k_{k})\cap\Omega''=\emptyset$, then it is enough
to choose $\k_{k+1}=\tilde{\k}_{k+1}$. Otherwise, applying Lemma~\ref{lem:finite set}
with $X=Z(\k_{1},\dots,\k_{k})$, we obtain that the set
\[
\{\k\in\A:\dim(Z(\k)\cap Z(\k_{1},\dots,\k_{k})\cap\Omega'')=\dim(Z(\k_{1},\dots,\k_{k})\cap\Omega'')\}
\]
is finite. Therefore, we can choose $\k_{k+1}\in\A$ such that
\[
\dim(Z(\k_{1},\dots,\k_{k+1})\cap\Omega'')<\dim(Z(\k_{1},\dots,\k_{k})\cap\Omega'')
\]
and $|\k_{k+1}-\tilde{\k}_{k+1}|<\epsilon$. Therefore, as we have
$\dim(Z(\k_{1})\cap\Omega'')\le d-1$, we obtain $\dim(Z(\k_{1},\dots,\k_{d+1})\cap\Omega'')<0$,
namely $Z(\k_{1},\dots,\k_{d+1})\cap\Omega''=\emptyset$. In other
words, $(\k_{1},\dots,\k_{d+1})\in G$. By construction, \eqref{eq:close}
is satisfied. This concludes the proof.
\end{proof}

\section{\label{sec:Conclusions}Conclusions}

In this work we have showed that, under the assumption of real analytic
coefficients, almost any $d+1$ frequencies in a fixed range give
the required constraints, where $d$ is the dimension of the ambient
space. The proof is based on the structure of analytic varieties,
and so the hypothesis of real analytic coefficients is crucial. To
prove (or disprove) this result under weaker hypothesis on the coefficients
a different approach is required.

While this result seems optimal for an a priori fixed number of boundary
conditions and for a somewhat arbitrary constraint function $F$,
it could be that less than $d+1$ frequencies are required if more
boundary conditions are allowed (e.g., a set of $d$ frequencies and
$d\times(d+1)$ boundary conditions to choose from). 

\bibliographystyle{abbrvurl}
\bibliography{alberti-capdeboscq}

\end{document}